\documentclass[preprint,12pt]{elsarticle}
\usepackage{amssymb}
\biboptions{sort&compress}

\usepackage[a4paper, portrait, margin=1.1811in]{geometry}
\usepackage[english]{babel}
\usepackage[utf8]{inputenc}
\usepackage[T1]{fontenc}
\usepackage{helvet}
\usepackage{etoolbox}
\usepackage{graphicx}
\usepackage{titlesec}
\usepackage{caption}
\usepackage{booktabs}
\usepackage{xcolor} 
\usepackage[colorlinks, citecolor=cyan]{hyperref}
\usepackage{appendix}

\usepackage{subfigure} 
\graphicspath{ {./images/} }
\usepackage{scrextend}
\usepackage{fancyhdr}
\usepackage{graphicx}

\usepackage{tikz} 
\usetikzlibrary{calc}

\usepackage{setspace} 
\usepackage{amsmath,amsthm,amssymb}
\usepackage{mathrsfs}
\usepackage{algorithm}
\usepackage{algorithmic}
\usepackage{color}

\usepackage{threeparttable}

\newtheorem{Thm}{Theorem}[section]
 
\newtheorem{Prop}[Thm]{Proposition}
\newtheorem{Lemma}[Thm]{Lemma}
\newtheorem{coro}[Thm]{Corollary}
\newtheorem{remark}[Thm]{Remark}
\newtheorem{Ex}[Thm]{Example}

\usepackage{lineno}

\begin{document}
	
	\begin{frontmatter}
		
		\title{Maximal signed bipartite graphs with totally disconnected graphs as star complements \tnoteref{1} }
		\tnotetext[1]{This research is partially supported by NSFC of China No. 12171088.}
		
		\author[]{Huiqun Jiang$^{*}$}
		\cortext[]{Corresponding author.}
		\ead{huiqun.jiang@foxmail.com}
		\author{Yue Liu}
		
		\address{School of Mathematics and Statistics, Fuzhou University, Fuzhou, 350108, Fujian, China.}
		
		\begin{abstract}
			Let $\dot{\mathscr{B}}\triangleq \dot{\mathscr{B}}_{\dot{H}, \mu}$ denote an arbitrary signed bipartite graph with $\dot{H}$ as a star complement for an eigenvalue $\mu$, where $\dot{H}$ is a totally disconnected graph of order $s$. In this paper, by using Hadamard and Conference matrices as tools, the maximum order of $\dot{\mathscr{B}}$ and the extremal graphs are studied. It is shown that $\dot{\mathscr{B}}$ exists if and only if $\mu^2$ is a positive integer. A formula of the maximum order of $\dot{\mathscr{B}}$ is given in the case of $\mu^2=p\times q$ such that $p$, $q$ are integers and there exists a $p$-order Hadamard or $(p+1)$-order Conference matrix. In particular, it is proved the maximum order of $\dot{\mathscr{B}}$ is $2s$ when either $q=1$, $s=c\mu^2=cp$ or $q=1$, $s=c(\mu^2+1)=c(p+1)$, $ c=1,2,3,\cdots$.
			Futhermore, some extremal graphs are characterized.
			
		\end{abstract}
		

	\end{frontmatter}
	
	\section{Introduction} 
	A \textit{signed graph} is a pair $(G, \sigma)$, where $G$ is a simple graph, called the \textit{underlying graph}, and $\sigma: E \to \{ -1,\,1\}$ is a mapping, called the \textit{sign function} or \textit{signature}. 
	In certain circumstances, it will be more convenient to deal with $(G, \sigma)$ in such a way that its set of negative edges is emphasized, in which case we will write $(G,\, \Sigma^{-})$ instead, where $\Sigma^{-}=\sigma^{-1} (-)$ denotes the set of negatives edges. 
	A simple graph $G$ can be regarded as a signed graph $(G,\emptyset)$.
	
	The \textit{Reconstruction Theorem} is well-known in the study of the spectrum of simple or signed graphs, which is firstly proved by Cvetkovi$\acute{\text{c} }$, Rowlison and Simi$\acute{\text{c} }$ in \cite{37}. The role of an individual eigenvalue in the structure of a simple or signed graph is revealed by the theorem.
	
	The \textit{star complement technique} \cite{35} is a procedure for constructing a simple or signed graph with a prescribed star complement for an eigenvalue to illustrate the use of the Reconstruction Theorem. In this procedure, there is a necessary condition that the eigenvalue does not appear in the spectrum of the prescribed star complement.
	In \cite{1}, by using this technique and the Kronecker product, Ramezani constructs some families of the signed graphs with only two opposite eigenvales.  
	
	The typical usages of this technique are to determine the maximum order of simple or signed graphs, and construct maximal simple or signed graphs.
	In \cite{3}, by using this technique, Ramezani, Rowlinson and Stani$\acute{\text{c} }$ prove that the maximum order $n$ of signed graphs is no more than $ \tbinom{n-k+2}{3}$, in which $k$ is the multiplicity of an eigenvalue $\mu\notin \{-1,0,1\}$.
	In \cite{4}, Yuan, Mao and Liu characterize maximal signed graphs with signed $C_3$ or $C_5$ as star complements for $-2$.
	In \cite{2}, Mulas and Stani$\acute{\text{c} }$ characterize maximal signed graphs $(G, \Sigma^{-})$ with any signed subgraphs whose spectrum lies in $(-2,2)$ as star complements, where $(G, \Sigma^{-})$ has the eigenvalues $\pm2$.
	
	Let $\dot{H}\triangleq (sK_1, \emptyset)$ be a totally disconnected signed graph of order $s$. In \cite{17}, Stani$\acute{\text{c} }$ studies signed graphs $(G,\, \Sigma^{-})$ with $\dot{H}$ as a star complement for an eigenvalue $\mu$. It is shown that if $s=\mu^2$, then the star set induces $\big((n-\mu^2) K_1, \emptyset \big)$ or $\big( K_{n-\mu^2},\, \emptyset \big)$, and the number of different eigenvalues of $(G,\, \Sigma^{-})$ is two or three, where $n$ is the order of $(G,\, \Sigma^{-})$. 
	
	It is easy to see $(G, \Sigma^{-})$ is bipartite in the case that the star set induces $\big( (n-\mu^2) K_1, \emptyset \big)$. Therefore, we study the following problem in this paper:
	
	\textbf{Problem $\mathcal{P}$:} 
	Let $\dot{H}$ be a totally disconnected graph of order $s$, $\mu$ be a non-zero real number. What is the maximum order of the signed bipartite graph $\dot{\mathscr{B}}$ such that $\dot{\mathscr{B}}$ has $\dot{H}$ as a star complement for the eigenvalue $\mu$ ?

	In Section 3, we show that $\dot{\mathscr{B}}$ exists if and only if $\mu^2$ is a positive integer and $s\geq \mu^2$ (see Proposition \ref{exist}).
	Since the adjacency matrix of $\dot{\mathscr{B}}$ can be orthogonal, then Hadamard and Conference matrices can be tools to study the problem $\mathcal{P}$. We can find some properties of Hadamard and Conference matrices in \cite{40}.

	A formula of the maximum order of $\dot{\mathscr{B}}$ is given in the case of $\mu^2=p\times q$ such that $p$, $q$ are integers and there exists a $p$-order Hadamard or $(p+1)$-order Conference matrix. In particular, the maximum order of $\dot{\mathscr{B}}$ is $2s$ in the case of $s=c\mu^2$ such that $\mu^2$ is the order of a Hadamard matrix, or $s=c(\mu^2+1)$ such that $\mu^2+1$ is the order of a Conference matrix (see Corollary \ref{coro}).
	
	The Section 2 contains terminologies and notations along with the Reconstruction Theorem. 
	In Section 4, some extremal graphs are characterized.

	\section{Preliminaries}	
	Let $(G, \Sigma^{-})$ be a signed graph of order $n$. If the vertices $u$ and $v$ of $(G, \Sigma^{-})$ are adjacent, then we write $u\sim v$. In particular, if they are adjacent by a positive (or negative) edge, then we write $u\stackrel{+}{\sim} v$ (or $u\stackrel{-}{\sim} v$ ). If $u$ and $v$ are not adjacent, then we write $u\nsim v$. The adjacency matrix of $(G, \Sigma^{-})$ is the $n \times n$ matrix $A_{(G, \Sigma^{-})}=(a_{ij})$, where $a_{ij}=1$ if $u\stackrel{+}{\sim} v$, $a_{ij}=-1$ if $u\stackrel{-}{\sim} v$ and $a_{ij}=0$ if $u\nsim v$. 
	
	Let $\mu$ be an eigenvalue of $(G, \Sigma^{-})$ with the multiplicity $k$. If $S$ is a subset of the set $V(G)$ such that $|S|=k$ and $\mu$ is not an eigenvalue of the induced subgraph $(G, \Sigma^{-})-S$, then $S$ is called a \textit{star set} for $\mu$ in $(G, \Sigma^{-})$ and the graph $(G, \Sigma^{-})-S$ (of order $n-k$) is called a \textit{star complement} for $\mu$ in $(G, \Sigma^{-})$. The properties of star sets and star complements for the corresponding eigenvalues can be found in \cite{37,19,35}. 
	
	The following result, called the Reconstruction Theorem, is fundamental to the theory of star complements.
	
	\begin{Thm} \label{reconstruction}
		\cite[Theorem 2.1]{3}
		Let $(G, \Sigma^{-}) $ be a graph with adjacency matrix
		\[\left(\begin{array}{cc}
			A_{S}	& B^T  \\ 
			B   & C
		\end{array}\right),\] 
		where $A_{S}$ is the $k\times k$ adjacency matrix of the subgraph induced by a vertex set $S$. 
		Then $S$ is a star set for $\mu$ in $(G, \Sigma^{-})$ if and only if $\mu$ is not an eigenvalue of $C$ and 
		\begin{equation}
			\label{one}
			\mu I-A_{S}=B^T(\mu I-C)^{-1}B.	
		\end{equation}
	\end{Thm}
	
	This result is initially formulated in the context of simple graphs. Since every simple graph can be interpreted as a signed graph, then it follows that Theorem \ref{reconstruction} is an extension of \cite[Theorem 7.4.1]{19}. In \cite{19,35}, Cvetkovi$\acute{\text{c} }$, Rowlinson and Simi$\acute{\text{c} }$ prove that star sets and star complements exist for any eigenvalues in simple or signed graphs.

	In order to solve the problem $\mathcal{P}$, it is necessary to compute an inverse of an inversable matrix forming as $\mu I-C$. In \cite[Proposition 5.2.1]{35}, there is the way to compute the matrix $(\mu I-C)^{-1}$. 
	
	Let $S$ be a star set of order $k$ for $\mu$ in $(G, \Sigma^{-})$. A bilinear form on $\mathbb{R}^{n-k}$ is defined as follows:
	\[ \langle x,\, y \rangle=x^{T} (\mu I-C)^{-1} y \quad (x,\, y\in \mathbb{R}^{n-k}).\] 
	Let $u$ be a vertex in $S$, $\mathbf{b}_u $ be a row of $B^T$ which  determines the neighbors of $u$ in the signed subgraph $(G, \Sigma^{-})-S$ and the signatures of associate edges. By Equation (\ref{one}), for every two vertices $u,\, v \in S $, we have
	\begin{equation} \label{R}
		\langle \mathbf{b}_u,\mathbf{b}_v \rangle=
		\left\{\begin{aligned}
			&\mu \quad &\text{if}\quad u=v ,\\
			&-1 \quad  &\text{if} \quad u \stackrel{+}{\sim} v ,\\
			&1 \quad   &\text{if} \quad u \stackrel{-}{\sim} v ,\\
			&0 \quad   &\text{if} \quad u \nsim v .
		\end{aligned}\right.
	\end{equation}
	If $\langle \mathbf{b}_u,\mathbf{b}_u \rangle=\mu$, then $u$ is said to be \textit{good} for $\mu$. If both $u$, $v$ are good for $\mu$, and $\langle \mathbf{b}_u,\mathbf{b}_v \rangle\in \{ -1,0, 1\}$, then $u$, $v$ are said to be \textit{compatible} for $\mu$. 
	
	Let $\dot{H}'$ be a signed graph with the adjacency matrix $C$, $\mu$ be a real number such that $\mu$ does not appear in the spectrum of $\dot{H}'$. Then $\mu$ does appear in the graph obtained by adding edges bewtween $\dot{H}'$ and a single vertex $u$, where the edges are allowed to have signatures by Equation (\ref{R}). It follows that $u$ is good for $\mu$ if and only if $\langle \mathbf{b}_u,\mathbf{b}_u \rangle=\mu$; if both $u$ and $v$ are good for $\mu$, then $u,v$ are compatible for $\mu$ if and only if $\langle \mathbf{b}_u,\mathbf{b}_v \rangle\in \{-1,0,1\}$. 
	
	Thus, in order to solve the problem $\mathcal{P}$, pairwise compatible vertices for $\mu$ are required as many as possible.

	\section{The maximum order of $\dot{\mathscr{B}}$}
	Let $\dot{H}$ be totally disconnected, $V(\dot{H})=\{ 1,2 \cdots,s \}$, $u $ be a single vertex such that $u \notin V(\dot{H})$, $\mathbf{b}_u =(b_{u1}, \cdots, b_{us} )^T$ be a $(0,\pm 1)$-vector such that $\mathbf{b}_u$ determines the neighbors of $u$ in $\dot{H}$ and the signatures of associate edges $u1, u2, \cdots,us$. By Equation (\ref{R}), we get
	\begin{equation} \label{1}
		\sum_{i=1}^{s} b_{ui} b_{vi}=
		\left\{\begin{aligned}
			&\mu^2 \quad &\text{if}\quad u=v ,\\
			&-\mu \quad &\text{if} \quad u \stackrel{+}{\sim} v ,\\
			&\mu \quad &\text{if} \quad u \stackrel{-}{\sim} v ,\\
			&0 \quad &\text{if} \quad u \nsim v .
		\end{aligned}\right.
	\end{equation}
	
	\begin{Prop} \label{exist}
		Let $\dot{H}$ be a totally disconnected graph of order $s $, $\mu$ be a non-zero real number, $\dot{\mathscr{B}}$ denote an arbitrary signed bipartite graph with $\dot{H}$ as a star complement for $\mu$. Then
		$\dot{\mathscr{B}}$ exists if and only if $\mu^2$ is a positive integer and $s\geq \mu^2$.
	\end{Prop}
	
	If $s=\mu^2$, then the following theorem is obtained in \cite{17}, which shows an inequality between $\mu^2$ and the order of $\dot{\mathscr{B}}$, and the spectrum of $\dot{\mathscr{B}}$.
	
	\begin{Thm} \cite[Theorem 3.11.]{17}
		If a signed graph $(G, \Sigma^{-})$ of order $n$ is decomposed into the star complement $\big( \mu^2 K_1, \emptyset \big)$ $(\text{for $\mu$})$ and $\big( (n-\mu^2) K_1, \emptyset \big)$, then $n\leq 2\mu^2$ and the spectrum of $(G, \Sigma^{-})$ is 
		\[ [\mu ^{n-\mu^2},0 ^{2\mu^2-n}, (-\mu) ^{n-\mu^2}].\]
	\end{Thm}
	
	This theorem reveals, if a signed graph $(G, \Sigma^{-})$ of order $n$ is decomposed into the star complement $\big( s K_1,\emptyset \big)$ $(\text{for $\mu$})$ and $\big( (n-s)K_1, \emptyset \big)$, then $n\leq 2s$ and the spectrum of $(G, \Sigma^{-})$ is 
	\[ [\mu ^{n-s},0 ^{2s-n}, (-\mu) ^{n-s}].\]
	It is easy to see the signed graph $(G, \Sigma^{-})$ is bipartite.
	
	Since $\dot{\mathscr{B}}$ is bipartite and $\dot{H}$ is totally disconnected, then for every two vertices $ u,v\notin V(\dot{H})$, $u \nsim v$, otherwise there is a induced signed graph $\big( K_3,\emptyset \big)$ in $\dot{\mathscr{B}}$. Thus the following theorem is obtained.
	
	\begin{Thm}
		Let $\dot{H}$ be a totally disconnected graph of order $s$, $\dot{\mathscr{B}}$ denote an arbitrary signed bipartite graph with $\dot{H}$ as a star complement for an eigenvalue $\mu$. If the order of $\dot{\mathscr{B}}$ is $n$, then $n\leq 2s$ and the spectrum of $\dot{\mathscr{B}}$ is 
		\[ [\mu ^{n-s},0 ^{2s-n}, (-\mu) ^{n-s}].\]
	\end{Thm}

	Let 
	\begin{equation}\label{adjacency}
		A_{\dot{\mathscr{B}} }=\left(\begin{array}{cc}
			\mathbf{0}	& B^T  \\ 
			B   & \mathbf{0}
		\end{array}\right).\end{equation}
	By Equations (\ref{1}) and (\ref{adjacency}), we get $B^TB=\mu^2I$. It is essential for the problem $\mathcal{P}$ to construct such $B$. 
	
	If $\mu^2=1$, then $B$ can be an identity matrix. Thus, the maximum order of $\dot{\mathscr{B}}$ is $2s$, where $\dot{\mathscr{B}}$ is a signed bipartite graph with $\dot{H}$ as a star complement for $\pm1$.
	\begin{Thm}
		Let $\dot{H}$ be a totally disconnected graph of order $s \,  (s\geq 2)$, $\dot{\mathscr{B}} $ denote an arbitrary signed bipartite graph with $\dot{H}$ as a star complement for an eigenvalue $\mu$. If $\mu^2=1$, then the maximum order of $\dot{\mathscr{B}}$ is $2s$.
	\end{Thm}
	
	If $s=\mu^2$, then $B$ is a square $(-1,1)$-matrix. An $n\times n$ $(-1,1)$-matrix $\mathcal{H}(n)$ with $\mathcal{H}(n)^T\mathcal{H}(n)=n I$ is called to be a \textit{Hadamard matrix}. 
	
	A Hadamard matrix is defined by Sylvester in \cite{41} and is studied further by Hadamard in \cite{42}. Hadamard conjectures that a Hadamard matrix of order $4n$ exists for every natural $n$ (\textit{Hadamard Conjecture}). This condition is necessary, and the sufficiency part is still an open problem. These Hadamard matrices are systematically studied by Paley in \cite{39}.

	It is shown the five following decompositions of a positive integer $n$ such that $n$ is the order of a Hadamard matrix:
	\begin{enumerate}
		\item[(1)] $n=2^p$, $p\geq 1$;
		\item[(2)] $n=2^p(q^h+1)$, $p\geq 2$, $h\geq 1$, $q$ is a prime;
		\item[(3)] $n=2^pq(q+1)$, $p\geq 2$, $q\equiv 3$ (mod 4), $q$ is a prime;
		\item[(4)] $n=q+1$, $q\equiv 3$ (mod 4), $q$ is a prime power;
		\item[(5)] $n=2(q+1)$, $q\equiv 1$ (mod 4), $q$ is a prime power.
	\end{enumerate}
	The cases of the first three decompositions are solved by Paley in \cite{39}, and the rest cases are solved by Ionin and Shrikhande in \cite[Corollary 4.3.26]{40}.
	
	Let $\mathcal{N}(\mathcal{H})$ be a subset of natural numbers such that every entry in $\mathcal{N}(\mathcal{H})$ is at least one of the five above decompositions. Firstly, we consider $\mu^2\in \mathcal{N}(\mathcal{H})$ and $s=\mu^2\geq 2$. 

	\begin{Lemma} \label{B1}
		Let $\mu^2 \, (\mu^2\geq 2)$ be a positive integer, $\mathbf{b}$ be a $(-1,1)$-vector of rows $\mu^2$, $B$ be a matrix consisted of $\mathbf{b}$'s such that $B^TB=\mu^2 I$. If $\mu^2\in \mathcal{N}(\mathcal{H})$, then
		the maximum number of colums of $B$ is $\mu^2$. 
	\end{Lemma}
	
	Secondly, we consider that $s=\mu^2$ and $\mu^2\in 2\mathbb{Z}^{+}-1$ or $\mu^2\in 2\mathbb{Z}^{+} \backslash 4\mathbb{Z}^{+}$. Let $\otimes$ be for the Kronecker product . 
	
	\begin{Lemma}\label{B}
		Let $\mu^2 \, (\mu^2\geq 2)$ be a positive integer, $\mathbf{b}$ be a $(-1,1)$-vector of rows $\mu^2$, $B$ be a matrix consisted of $\mathbf{b}$'s such that $B^TB=\mu^2 I$. If $\mu^2\in 2\mathbb{Z}^{+}+1$ or $\mu^2\in 2\mathbb{Z}^{+} \backslash 4\mathbb{Z}^{+}$, then the maximum number of colums of $B$ is $1$ or $2$.
	\end{Lemma}
	\begin{proof}
		Let $\mu^2=2^p(2q+1)$, where $p=0,1$ and $q=1,2,3,\cdots$. If $p=0$, then $B$ is a vector. Otherwise, there is at least one $0$ in the second column of $B$, which is a contradiction that $B$ is a $(-1,1)$-matrix. Thus the maximum number of colums of $B$ is $1$. 
		
		Let
		\[ \mathcal{H}(2)=\left(\begin{array}{cc}
			1&1  \\
			1&-1
		\end{array}\right),\quad 
		B=\mathcal{H}(2) \otimes \mathbf{j}_{2q+1} ,\quad 
		\mathbf{y}=	\mathbf{y}_1  \otimes \mathbf{y}_2\]
		be a $(-1,1)$-vector such that $B^T\mathbf{y}=\mathbf{0}$, where the number of rows of $\mathbf{y}_1$ (or $\mathbf{y}_2$) is $2$ (or $2q+1$). Then 
		\[\left(\mathcal{H}(2)^T \mathbf{y}_1 \right) \otimes \left(\mathbf{j}^T \mathbf{y}_2 \right) =\mathbf{0},\]
		and $\mathcal{H}(2)^T \mathbf{y}_1=\mathbf{0}$ or $\mathbf{j}^T \mathbf{y}_2 =0$, which is a contradiction. Thus the maximum number of colums of $B$ is $2$.
	\end{proof}
	
	By Lemma \ref{B1} and \ref{B}, the following theorem is obtained.
	\begin{Thm}\label{B3}
		Let $\dot{H}$ be a totally disconnected graph of order $s$, $\dot{\mathscr{B}}$ denote an arbitrary signed bipartite graph with $\dot{H}$ as a star complement for an eigenvalue $\mu$, $n$ be the maximum order of $\dot{\mathscr{B}}$. If $s=\mu^2\geq 2$, then
		\[n=\left\{\begin{aligned}
			&s+1&\, &\text{$s\in 2\mathbb{Z}^{+}+1 $},\\
			&s+2&\, &\text{$s\in 2\mathbb{Z}^{+} \backslash 4\mathbb{Z}^{+}$},\\
			&2s&\, &s\in \mathcal{N}(\mathcal{H}).\\
		\end{aligned}
		\right.\] 
	\end{Thm}
	
	If $s=\mu^2+1$, then we consider whether there is a $(0,\pm 1)$-matrix $B$ such that $B^TB=(s-1)I$ or not. An $n\times n$ square matrix $\mathcal{C}(n)$ with $\mathcal{C}(n)^T \mathcal{C}(n)=(n-1) I$, in which all the diagonal entries are 0 and all off-diagonal entries are $\pm 1$, is said to be a \textit{Conference matrix}. In \cite[Proposition 4.3.5.]{40}, 
	it is shown that there is a Conference matrix $\mathcal{C}(n)$ in the case that $n-1$ is an odd prime power (the odd prime power can be 1). 
	
	Let $\mathcal{N}(\mathcal{C})$ be a set of odd prime powers. Thirdly, we consider that $\mu^2\in \mathcal{N}(\mathcal{C})$ and $s=\mu^2+1$. 
	\begin{Lemma} \label{B2}
		Let $\mu^2\in \mathcal{N}(\mathcal{C})$ with $\mu^2\geq 2$, $\mathbf{b}$ be a $(0,\pm 1)$-vector of rows $\mu^2+1$, in which the number of zero is $1$. If there is a matrix $B$ consisted of $\mathbf{b}$'s such that $B^TB=\mu^2 I$, then the maximum number of colums of $B$ is $\mu^2+1$.
	\end{Lemma}
	
	\begin{Thm}\label{B4}
		Let $\dot{H}$ be a totally disconnected graph of order $s$, $\dot{\mathscr{B}}$ denote an arbitrary signed bipartite graph with $\dot{H}$ as a star complement for an eigenvalue $\mu$, $n$ be the maximum order of $\dot{\mathscr{B}}$. If $s=\mu^2+1\geq 3$, then
		\[n=\left\{\begin{aligned}
			&\mu^2+3&\, &\text{$\mu^2\in 2\mathbb{Z}^{+} \backslash 4\mathbb{Z}^{+}$},\\
			&2\mu^2+1&\, &\mu^2\in \mathcal{N}(\mathcal{H}),\\
			&2\mu^2+2&\, &\text{$\mu^2\in \mathcal{N}(\mathcal{C})$}.\\
		\end{aligned}
		\right.\] 
	\end{Thm}
	\begin{proof}
		If $\mu^2\in 2\mathbb{Z}^{+} \backslash 4\mathbb{Z}^{+}$, then $n\geq \mu^2+3$ by Theorem \ref{B3}.
		We suppose that there are two vertices $u$, $v$ in $\dot{\mathscr{B}}$ such that the neighbors of $u$ are different from ones of $v$. Then the number of the common neighbors between $u$ and $v$ is $\mu^2-1$. Let $\mathbf{b}_u$ is a $(0,\pm 1)$-vector of order $s$ such that $\mathbf{b}_u$ determines the neighbors of $u$ in $\dot{H}$ and the signatures of associate edges. Since $\mu^2$ is even, then $\mu^2-1$ is odd and $\mathbf{b}_u^T \mathbf{b}_v\neq \mathbf{0}$, which is a contradiction. Thus $n=\mu^2+3$.
		
		If $\mu^2\in \mathcal{N}(\mathcal{H})$, then $n=2\mu^2+1$ by Theorem \ref{B3}. If $s\in \mathcal{N}(\mathcal{C})$, then $n=2\mu^2+2$ by Lemma \ref{B2}.
	\end{proof}

	Finally, we consider that $\mu^2=p\times q$ such that $p$, $q$ are integers and there exists a $p$-order Hadamard or $(p+1)$-order Conference matrix.
	\begin{Thm}\label{C1}
		Let $\dot{H}$ be a totally disconnected graph of order $s$, $\dot{\mathscr{B}}$ denote an arbitrary signed bipartite graph with $\dot{H}$ as a star complement for an eigenvalue $\mu$, $n$ be the maximum order of $\dot{\mathscr{B}}$, where there are positive integers $p$ and $q$ such that $\mu^2=pq\geq 2$.
		\begin{enumerate}
			\item[(1)] If $p,\, q\in \mathcal{N}(\mathcal{H})$ and $cpq\leq s< (c+1)pq$, then $n=s+cpq$, where $c=1,2,3,\cdots$.
			\item[(2)] If $q\in \mathcal{N}(\mathcal{C})$, then Table \ref{n-s} holds.
			\begin{table}[H]
				\caption{The maximum order $n$ of $\dot{\mathscr{B}}$, where $\mu^2=pq$, $q\in \mathcal{N}(\mathcal{C})$ and $c=1,2,3,\cdots$. } 
				\label{n-s} 
				\centering
				\renewcommand\arraystretch{1.5} 
				\resizebox{\linewidth}{!}{
					\begin{tabular}{|c|c|c|c|}
						\hline
						cases &$p$ & $s$  & $n-s$ \\
						\hline
						(2.1)&$2$ &  $2q$ or $2q+1$& $2$\\
						\hline
						(2.2)&$p\in \mathcal{N}(\mathcal{H})$ &  $cp(q+1)\leq s<(c+1)p(q+1)$& $cp(q+1)$\\
						\hline
						(2.3)&$p\in \mathcal{N}(\mathcal{C})$ &  $ c(p+1)(q+1)\leq s< (c+1)(p+1)(q+1)$& $c(p+1)(q+1)$\\
						\hline
				\end{tabular} }
			\end{table}
		\end{enumerate}
	\end{Thm}
	\begin{proof}
		Since $p,\, q\in \mathcal{N}(\mathcal{H})$, then there are Hadamard matrices $\mathcal{H}(p)$ and $\mathcal{H}(q)$. If $s=pq$, then $B=\mathcal{H}(p) \otimes \mathcal{H}(q)$. 
		Let $\mathbf{y}=\mathbf{y}_1 \otimes \mathbf{y}_2$ be a $(-1,1)$-vector such that $B^T\mathbf{y}=\mathbf{0}$, where the number of rows of $\mathbf{y}_1$ (or $\mathbf{y}_2$) is $p$ (or $q$). Then 
		\[\left(\mathcal{H}(p)^T \mathbf{y}_1 \right) \otimes \left(\mathcal{H}(q)^T \mathbf{y}_2 \right) =\mathbf{0},\]
		and there is a zero entry in $\mathcal{H}(p)^T \mathbf{y}_1$ or $\mathcal{H}(q)^T \mathbf{y}_2$, which is a contradiction. Thus $n=s+pq$. If $cpq\leq s<(c+1)pq$, then 
		\[B^T=\left( I_c\otimes \mathcal{H}(p)^T \otimes \mathcal{H}(q)^T \quad \mathbf{0} \right),\] 
		and $n=s+cpq$. Thus the result $(1)$ holds.
		
		Since $q\in \mathcal{N}(\mathcal{C})$, then there is a Conference matrix $\mathcal{C}(q+1)$. If $s=\mu^2=2q$, then $n=2q+3=s+2$ by Theorem \ref{B3}. If $s=\mu^2+1=2q+1$, then $n= s+2$ by Theorem \ref{B4}. Thus the result $(2.1)$ holds.
		
		Since $p\in \mathcal{N}(\mathcal{H})$, then there is a Hadamard matrix $\mathcal{H}(p)$. If $s=p(q+1)$, then 
		$B=\mathcal{H}(p) \otimes \mathcal{C}(q+1)$.
		Let $\mathbf{x}_1$ be a $(0, \pm 1)$-vector of rows $p$, $\mathbf{x}_2$ be a $(0, \pm 1)$-vector of rows $q+1$, $\mathbf{x}=\mathbf{x}_1 \otimes \mathbf{x}_2$ be a vector such that $B^T\mathbf{x}=\mathbf{0}$ and there is only one $0$ either in $\mathbf{x}_1$ or in $\mathbf{x}_2$. Then \[\left(\mathcal{H}(p)^T \mathbf{x}_1 \right) \otimes \left(\mathcal{C}(q+1)^T \mathbf{x}_2 \right) =\mathbf{0}.\]
		If there is only one $0$ in $\mathbf{x}_1$, then there are no zero entries in $\mathcal{H}(p)^T \mathbf{x}_1$ or $\mathcal{C}(q+1)^T \mathbf{x}_2$, which is a contradiction.
		If there is only one $0$ in $\mathbf{x}_2$, then there is a zero entry in $\mathcal{H}(p)^T \mathbf{x}_1$ or $\mathcal{C}(q+1)^T \mathbf{x}_2$, which is a contradiction. 
		Thus $n=s+p(q+1)$. If $cp(q+1)\leq s<(c+1)p(q+1)$, then 
		\[B^T=\left( I_c\otimes \mathcal{H}(p)^T \otimes \mathcal{C}(q+1)^T \quad \mathbf{0} \right),\] 
		and $n=s+cp(q+1)$. Thus the result $(2.2)$ holds.
			
		Since $p\in \mathcal{N}(\mathcal{C})$, then there is a Conference matrix $\mathcal{C}(p+1)$. If $s=(p+1)(q+1)$, then $B=\mathcal{C}(p+1) \otimes \mathcal{C}(q+1)$. Let $\mathbf{y}_1$ be a $(0, \pm 1)$-vector of rows $p+1$, $\mathbf{y}_2$ be a $(0, \pm 1)$-vector of rows $q+1$, $\mathbf{y}=\mathbf{y}_1 \otimes  \mathbf{y}_2$ be a vector such that $B^T\mathbf{y}=\mathbf{0}$ and there is either only one $0$ in both $\mathbf{y}_1$ and $\mathbf{y}_2$ or two $0$'s in $\mathbf{y}_1$. Then 
		\[\left( \mathcal{C}(p+1)^T \mathbf{y}_1 \right) \otimes \left( \mathcal{C}(q+1)^T \mathbf{y}_2 \right) =\mathbf{0}.\]
		If there is only one $0$ in $\mathbf{y}_1$ and $\mathbf{y}_2$, then there is a zero entry in $\mathcal{C}(p+1)^T \mathbf{y}_1$ or $\mathcal{C}(q+1)^T \mathbf{y}_2$, which is a contradiction. If there are only two $0$'s in $\mathbf{y}_1$, then not all entries in $ \mathcal{C}(p+1)^T \mathbf{y}_1$ are $0$'s and there are no zero entries in $\mathcal{C}(q+1)^T \mathbf{y}_2$, which is a contradiction.
		Thus $n=s+(p+1)(q+1)$. If $c(p+1)(q+1)\leq s< (c+1)(p+1)(q+1)$, then \[B^T=( I_c\otimes \mathcal{C}(p+1)^T \otimes \mathcal{C}(q+1)^T \quad \mathbf{0}) ,\] 
		and $n=s+c(p+1)(q+1)$. Thus the result $(2.3)$ holds.
	\end{proof}
	
	\begin{coro}\label{coro}
		Let $\dot{H}$ be a totally disconnected graph of order $s$, $\dot{\mathscr{B}}$ denote an arbitrary signed bipartite graph with $\dot{H}$ as a star complement for an eigenvalue $\mu$ $(\mu^2\geq 2)$, $n$ be the maximum order of $\dot{\mathscr{B}}$.
		\begin{enumerate}
			\item[(1)]  If $\mu^2\in \mathcal{N}(\mathcal{H})$ and $c\mu^2\leq s< (c+1)\mu^2$, then $n=s+c\mu^2$.
			\item[(2)] If $\mu^2\in \mathcal{N}(\mathcal{C})$ and $c(\mu^2+1)\leq s<(c+1)(\mu^2+1)$, then $n=s+c(\mu^2+1)$.
		\end{enumerate}
	\end{coro}

	\begin{Thm}\label{C2}
		Let $\dot{H}$ be a totally disconnected graph of order $s$, $\dot{\mathscr{B}}$ denote an arbitrary signed bipartite graph with $\dot{H}$ as a star complement for an eigenvalue $\mu$, $n$ be the maximum order of $\dot{\mathscr{B}}$, where there are positive integers $p$ and $q$ ($p> q$) such that $\mu^2=p q\geq 2 $.
		\begin{enumerate}
			\item[(1)] If $p\in \mathcal{N}(\mathcal{H})$, $q\in \mathcal{N}(\mathcal{C})$ and $pq\leq s< p(q+1)$, then $s+p\leq n\leq 2s$.
			\item[(2)] If $p,\, q\in \mathcal{N}(\mathcal{C})$ and $pq\leq s< (p+1)q$, then $s+1\leq n\leq 2s$.
			\item[(3)] If $p,\, q\in \mathcal{N}(\mathcal{C})$ and $(p+1)q\leq s< (p+1)(q+1)$, then $s+p+1\leq n\leq 2s$.
		\end{enumerate}
	\end{Thm}
	\begin{proof}
		Since $p\in \mathcal{N}(\mathcal{H})$, $q\in \mathcal{N}(\mathcal{C})$, then there is a Hadamard matrix $\mathcal{H}(p)$ and a Conference matrix $\mathcal{C}(q+1)$.
		If $s=p q$, then $B=\mathcal{H}(p)\otimes \mathbf{j}_{q}$. Let $\mathbf{x}=\mathbf{x}_1 \otimes \mathbf{x}_2$ be a $(-1,1)$-vector such that $B^T\mathbf{x}=\mathbf{0}$, where the number of rows of $\mathbf{x}_1$ (or $\mathbf{x}_2$) is $p$ (or $q$). Then 
		\[\left(\mathcal{H}(p)^T \mathbf{x}_1 \right) \otimes \left(\mathbf{j}^T \mathbf{x}_2 \right) =\mathbf{0},\]
		and there is a zero entry in $\mathcal{H}(p)^T \mathbf{x}_1$ or $\mathbf{j}^T \mathbf{x}_2 =0$, which is a contradiction. Thus $s+p\leq n\leq 2s$.
		
		If $s=\mu^2=p q$ and $p,\, q\in \mathcal{H}(\mathcal{C})$, then $n=s+1$ by Lemma \ref{B3}. If $pq\leq s<(p+1)q$, then $s+1\leq n\leq 2s$.
		
		Since $p\in \mathcal{N}(\mathcal{C})$, then there is a Conference matrix $\mathcal{C}(p+1)$. If $s=(p+1)q$, then $B=\mathcal{C}(p+1)\otimes \mathbf{j}_{q}$. Let $\mathbf{y}_1$ be a $(0,\pm 1)$-vector of rows $p+1$, $\mathbf{y}_1$ be a $(0,\pm 1)$-vector of rows $q$, $\mathbf{y}=\mathbf{y}_1 \otimes \mathbf{y}_2$ be a vector such that $B^T\mathbf{y}=\mathbf{0}$ and there is only one $0$ either in $\mathbf{y}_1$ or in $\mathbf{y}_2$. Then 
		\[\left(\mathcal{C}(p+1)^T \mathbf{y}_1 \right) \otimes \left( \mathbf{j}^T \mathbf{y}_2 \right) =\mathbf{0}.\]
		If there is only one $0$ in $\mathbf{y}_1$, then there is a zero entry in $\mathcal{C}(p+1)^T \mathbf{y}_1$ and $\mathbf{j}^T \mathbf{y}_2 \neq 0$, which is a contradiction. If there is only one $0$ in $\mathbf{y}_2$, then there are no zero entries in $\mathcal{C}(p+1)^T \mathbf{y}_1$ and $\mathbf{j}^T \mathbf{y}_2 \neq 0$, which is a contradiction. Thus $s+p+1\leq n\leq 2s$. 
	\end{proof}

	\section{The maximal signed bipartite graph $ \dot{\mathscr{B}}_m$}	
	Let $(G_1, \Sigma_1^{-})$ and $(G_2, \Sigma^{-}_2)$ be two signed graphs. If there exists a permutation $(0,\,1)$-matrix $P$ such that $A_{(G_2, \Sigma^{-}_2)}=P^{-1} A_{(G_1, \Sigma^{-}_1)} P$, then $(G_1, \Sigma^{-}_1)$ and $(G_2, \Sigma^{-}_2)$ are said to be \textit{isomorphic}. Similarly, if there is a diagonal $(-1,1)$-matrix $D$ such that $A_{ (G_2, \Sigma^{-}_2) }=D^{-1} A_{ (G_1, \Sigma^{-}_1) } D$, then $(G_1, \Sigma^{-}_1)$ and $(G_2, \Sigma^{-}_2)$ are said to be \textit{switching equivalent}. Isomorphism and switching equivalence are equivalence relations preserving the eigenvalues. 
	
	If $(G_1, \Sigma^{-}_1)$ is isomorphic to a switching equivalence of $(G_2, \Sigma^{-}_2)$, then $(G_1, \Sigma^{-}_1)$ and $(G_2, \Sigma^{-}_2)$ are said to be \textit{switching isomorphic}, denoted by $(G_1, \Sigma^{-}_1)\cong^s (G_2, \Sigma^{-}_2)$. Switching isomorphism is also an equivalence relation preserving the eigenvalues.    
	
	Let $\dot{H}$ be a totally disconnected graph of order $s$, $\dot{\mathscr{B}}_m$ is a maximal signed bipartite graph with $\dot{H}$ as a star complement for an eigenvalue $\mu$.
	
	\begin{Thm}
		Let $\dot{H}$ be a totally disconnected graph of order $s$, $\dot{\mathscr{B}}_m$ denote the maximal signed bipartite graph with $\dot{H}$ as a star complement for an eigenvalue $\mu$. If $\mu^2=1$, then $\dot{\mathscr{B}}_m \cong^s(sK_{2},\emptyset)$.
	\end{Thm}
	\begin{proof}
		Let $S$ be a vertex subset of $\dot{\mathscr{B}} _m$ such that $\dot{\mathscr{B}}_m-S=\dot{H}$. Since $\mu=\pm 1$, then by Equation (\ref{1}), we obtain that each vertex in $S$ is adjacent with only one vertex in $\dot{H}$, and each vertex in $\dot{H}$ does not have common neighbors in $S$. Thus $\dot{\mathscr{B}}_m \cong^s (sK_{2},\emptyset)$. 
	\end{proof}	
	
	Now we consider $s=\mu^2\geq 2$ or $s=\mu^2+1$.
	
	\begin{Ex}
		Let $\dot{H}$ be a totally disconnected graph of order $\mu^2$.
		If $\mu^2=2$, then $(K_{2,2},E(K_2) )$ is the maximal signed bipartite graph with $\dot{H}$ as a star complement for $\pm \sqrt{2}$. If $\mu^2=3$, then $\big( K_{1,3},\emptyset \big)$ is the maximal signed bipartite graph with $\dot{H}$ as a star complement for $\pm \sqrt{3}$. If $\mu^2=4$, then $\big( K_{4,4},E(C_6) \big)$ is the maximal signed bipartite graph with $\dot{H}$ as a star complement for $\pm 2$.
	\end{Ex}
	
	A simple graph is said to be \textit{biregular} if its vertex degrees assume exactly two different values. Let $K_{n,n}\backslash n K_2$ be a simple graph obtained by deleting $n$ edges without common vertices in $K_{n,n}$. 
	
	\begin{Ex}
		Let $\dot{H}$ be a totally disconnected graph of order $\mu^2+1$.
		If $\mu^2=2$, then $(K_{1,3}\cup K_1,\emptyset)$ is the maximal signed bipartite graph with $\dot{H}$ as a star complement for $\pm \sqrt{2}$. If $\mu^2=3$, then $\big( K_{4,4} \backslash {4K_2}, E(P_4\cup K_2) \big)$ is the maximal graph with $\dot{H}$ as a star complement for $\pm \sqrt{3}$, where $P_4$ is a path of length $4$. If $\mu^2=5$, then $\big( K_{6,6} \backslash {6K_2}, E(BR) \big)$ is the maximal graph with $\dot{H}$ as a star complement for $\pm \sqrt{5}$, where the simple graph $BR$ (see Figure \ref{Sigma}) is biregular.
	\end{Ex}
	
	\begin{figure}[h]
		\centering
		\begin{tikzpicture}[scale=0.4]		
			\node[circle, fill=black!60,inner sep=1pt] (a1) at (0,0) {};
			\node[circle, fill=black!60,inner sep=1pt] (a2) at (2,0) {};
			\node[circle, fill=black!60,inner sep=1pt] (a3) at (4,0) {};
			\node[circle, fill=black!60,inner sep=1pt] (a4) at (6,0) {};
			\node[circle, fill=black!60,inner sep=1pt] (a5) at (8,0) {};
			\node[circle, fill=black!60,inner sep=1pt] (a6) at (10,0) {};
			\node[circle, fill=black!60,inner sep=1pt] (b1) at (0,4) {};
			\node[circle, fill=black!60,inner sep=1pt] (b2) at (2,4) {}; 
			\node[circle, fill=black!60,inner sep=1pt] (b3) at (4,4) {}; 
			\node[circle, fill=black!60,inner sep=1pt] (b4) at (6,4) {}; 
			\node[circle, fill=black!60,inner sep=1pt] (b5) at (8,4) {}; 
			\node[circle, fill=black!60,inner sep=1pt] (b6) at (10,4) {}; 
			\draw[line width=0.8pt,-] (a1) --  (b2);
			\draw[line width=0.8pt,-] (a1) --  (b5);
			\draw[line width=0.8pt,-] (a1) --  (b6);
			\draw[line width=0.8pt,-] (a2) --  (b1);
			\draw[line width=0.8pt,-] (a2) --  (b4);
			\draw[line width=0.8pt,-] (a2) --  (b6);
			\draw[line width=0.8pt,-] (a3) --  (b4);
			\draw[line width=0.8pt,-] (a3) --  (b5);
			\draw[line width=0.8pt,-] (a3) --  (b6);
			\draw[line width=0.8pt,-] (a4) --  (b1);
			\draw[line width=0.8pt,-] (a4) --  (b5);
			\draw[line width=0.8pt,-] (a5) --  (b2);
			\draw[line width=0.8pt,-] (a5) --  (b4);	
		\end{tikzpicture}
		\caption{The biregular graph $BR$ is a induced subgraph in $K_{6,6}\backslash {6K_2}$.}
		\label{Sigma}
	\end{figure}
	
	If $s=\mu^2\in \mathcal{N}(\mathcal{H})$, then the order of $\dot{\mathscr{B}}_m$ is $2s$ by Theorem \ref{B3}, and $\dot{\mathscr{B}}_m$ has only two opposite eigenvalue. If $s=\mu^2+1$, then the maximum order of $\dot{\mathscr{B}}_m$ is $2s+2$ by Theorem \ref{B4}, and $\dot{\mathscr{B}}_m$ also has only two opposite eigenvalue. 
	
	In \cite{1}, Ramezani proves that the signed graphs with just two distinct eigenvalues are signed strongly regular graphs. In \cite{43}, Stani$\acute{\text{c} }$ shows certain structural and spectral properties of signed strongly regular graphs, which is also bipartite.
	
	Let the signature of a walk be determined by the product of signatures of edges, $u$ and $v$ be two vertices in a signed graph, $w_2(u,v)$ be the difference between the numbers of positive and negative walks traversing along $2$ edges between $u$ and $v$. A signed graph $(G, \Sigma^{-})$ is said to be strongly regular (for short, $SRG^s (n,r,a,b,c)$) whenever it is regular and satisfies the following four conditions: 
	\begin{enumerate}
		\item[(1)] $(G, \Sigma^{-})$ is neither homogeneous complete nor totally disconnected;
		\item[(2)] there exists $a \in\mathbb{Z}$ such that $w_2(u,v) = a$, for all $u\stackrel{+}{\sim} v$;
		\item[(3)] there exists $b \in\mathbb{Z}$ such that $w_2(u,v) = b$, for all $u\stackrel{-}{\sim} v$;
		\item[(4)] there exists $c\in\mathbb{Z}$ such that $w_2(u,v) = c$, for all $u\nsim v$.
	\end{enumerate}
	Thus 
	\begin{equation} \label{SRG}
		A_{(G, \Sigma^{-})}^2=\frac{a}{2}( A_{(G, \Sigma^{-})} +A_{G})-\frac{b}{2}( A_{(G, \Sigma^{-})}-A_{G} )+
		cA_{ \overline{G} }+rI,
	\end{equation}
	where $\overline{G}$ is for the complement of a simple graph $G$. We use the notation $SRG^s(2n)$ to denote as $SRG^s(2n,n,0,0,0)$.
	
	\begin{Thm}\label{G1}
		Let $\dot{H}$ be a totally disconnected graph of order $s$, $\dot{\mathscr{B}}_m$ denote the maximal signed bipartite graph with $\dot{H}$ as a star complement for an eigenvalue $\mu$, where $s=\mu^2\geq 2$.
		\begin{enumerate}
			\item[(1)] If $s\in 2\mathbb{Z}^{+}+1$, then $\dot{\mathscr{B}}_m \cong^s (K_{1,s}, \, \emptyset)$.
			\item[(2)] If $s\in 2\mathbb{Z}^{+} \backslash 4\mathbb{Z}^{+}$, then $\dot{\mathscr{B}}_m \cong^s \big( K_{2,s}, \, E(K_{1,\frac{s}{2}}) \big)$.
			\item[(3)] If $s\in \mathcal{N}(\mathcal{H})$, then 
			$\dot{\mathscr{B}}_m \cong^s SRG^{s}(2s)$.
		\end{enumerate} 
	\end{Thm}
	\begin{proof}
		Since $\dot{\mathscr{B}}_m$ is bipartite and $\dot{H}$ is totally disconnected, then for every two vertices $u,v\notin V(\dot{H})$, $u \nsim v$, otherwise there is a induced signed graph $(K_3,\emptyset)$ in $\dot{\mathscr{B}}_m$. By Equations (\ref{1}) and (\ref{adjacency}), we get $B^TB=sI$.
		
		If $s\in 2\mathbb{Z}^{+}-1$ or $s\in 2\mathbb{Z}^{+} \backslash 4\mathbb{Z}^{+}$, then the results (1) and (2) hold by Lemma \ref{B1}. If $s\in \mathcal{N}(\mathcal{H})$, then $A_{\dot{\mathscr{B}}_m}^2=sI$ and $\dot{\mathscr{B}}_m$ is a signed strongly regular graph. Let $n$ be the order of $\dot{\mathscr{B}}_m$, $r,\, a,\, b,\, c$ be integers satisfying Equation (\ref{SRG}). Then $n=2s$, $r=s$ and $a=b=c=0$. Thus the result $(3)$ holds.
	\end{proof}
	
	\begin{Thm}\label{G2}
		Let $\dot{H}$ be a totally disconnected graph of order $s$, $\dot{\mathscr{B}}_m$ denote the maximal signed bipartite graph with $\dot{H}$ as a star complement for an eigenvaule $\mu$, where $s=\mu^2+1\geq 3$.
		\begin{enumerate}
			\item[(1)] If $\mu^2 \in 2\mathbb{Z}^{+} \backslash 4\mathbb{Z}^{+}$, then $\dot{\mathscr{B}}_m \cong^s \big( K_{2,\mu^2}\cup K_1, \, E(K_{1,\frac{\mu^2}{2}}) \big)$.
			\item[(2)] If $\mu^2\in \mathcal{N}(\mathcal{H})$, then 
			$\dot{\mathscr{B}}_m \cong^s SRG^{s}(2\mu^2)\cup (K_1, \, \emptyset)$.
			\item[(3)] If $\mu^2\in \mathcal{N}(\mathcal{C})$, then 
			$\dot{\mathscr{B}}_m \cong^s SRG^{s}(2\mu^2+2)$.
		\end{enumerate}
	\end{Thm}
	\begin{proof}
		By Equation (\ref{adjacency}), we get $B^TB=\mu^2I$. If $\mu^2 \in (2\mathbb{Z}^{+} \backslash 4\mathbb{Z}^{+}) \cup\mathcal{N}(\mathcal{H})$, then every two vertices in $\dot{\mathscr{B}}_m$ have the same neighbors. Otherwise the inner product between a $(-1,1)$-vector with odd rows and an all-one vector is zero, which is a contradiction. The results $(1)$ and $(2)$ hold by Theorem \ref{G1}.
		If $\mu^2\in \mathcal{N}(\mathcal{C})$, then $A_{\dot{\mathscr{B}}_m}^2=\mu^2I$,
		and $\dot{\mathscr{B}}_m$ is a signed strongly regular graph. Let $n$ be the order of $\dot{\mathscr{B}}_m$, $r,\, a,\, b,\, c$ be integers satisfying Equation (\ref{SRG}), then
		$n=2\mu^2+2$, $r=s+1$ and $a=b=c=0$ by Lemma \ref{B2}. Thus the result $(3)$ holds. 
	\end{proof}
	
	\begin{remark}
		Let $SRG^{s}(2s)$ is bipartite. If $s\in \mathcal{N}(\mathcal{H})$, then the underlying graph of $SRG^{s}(2s)$ is $K_{s,s}$. If $s\in \mathcal{N}(\mathcal{C})$, then the underlying graph of $SRG^{s}(2s)$ is $K_{s,s}\backslash s K_2$. 
	\end{remark}
	
	Let $(G_1, \Sigma^{-}_1)$ and $(G_2, \Sigma^{-}_2)$ be two signed bipartite graphs with the adjacency matrices 
	\[\left(\begin{array}{cc}
		\mathbf{0}	& B_1^T  \\ 
		B_1   & \mathbf{0}
	\end{array}\right) \quad \text{and} \quad 
	\left(\begin{array}{cc}
	\mathbf{0}	& B_2^T  \\ 
	B_2   & \mathbf{0}
	\end{array}\right), \]
	respectively. The signed graph $(G_1, \Sigma^{-}_1) \dot{\otimes} (G_2, \Sigma^{-}_2)$ is defined as the signed bipartite graph with the adjacency matrix 
	\[A_{(G_1, \Sigma^{-}_1) } \dot{\otimes} A_{(G_2, \Sigma^{-}_2)}=\left(\begin{array}{cc}
	\mathbf{0}	& B_1^T\otimes B_2^T  \\ 
	B_1 \otimes B_2  & \mathbf{0}
	\end{array}\right).\]
	Let $c(G_1, \Sigma^{-}_1)$ be for the disjoint union of $c$ copies of $(G_1, \Sigma^{-}_1)$.
	
	By Theorem \ref{C1}, \ref{G1} and \ref{G2}, the following result is obtained.
	\begin{Thm}
		Let $\dot{H}$ be a totally disconnected graph of order $s$, $\dot{\mathscr{B}}_m$ denote the maximal signed bipartite graph with $\dot{H}$ as a star complement for an eigenvalue $\mu$, where there are positive integers $p$ and $q$ such that $\mu^2=p q\geq2$.
		\begin{enumerate}
			\item[(1)] If $p,q\in \mathcal{N}(\mathcal{H})$ and $cpq\leq s<(c+1)pq \, (c=1,2,3,\cdots)$, then 
			\[\dot{\mathscr{B}}_m \cong^s \big( c \, SRG^s(2p)\dot{\otimes} SRG^s(2q) \big)
			\cup \big( (s-cpq )K_1,\emptyset \big).\]
			\item[(2)] 	If $q\in \mathcal{N}(\mathcal{C})$, then Table \ref{G} holds.	
			\begin{table}[H]
				\caption{The maximal signed bipartite graph $\dot{\mathscr{B}}_m$, where $\mu^2=pq$, $q\in \mathcal{N}(\mathcal{C})$ and $c=1,2,3,\cdots$.} 
				\label{G} 
				\centering
				\renewcommand\arraystretch{1.5} 
				\resizebox{\linewidth}{!}{
					\begin{tabular}{|c|c|c|@{}}
						\hline
						$p$ & $s$  & $\dot{\mathscr{B}}_m$ is switching isomorphic to \\
						\hline
						2 &  $2q$ or $2q+1$ & $\left( K_{2,2q}\cup (s-2q) K_1, E(K_{1,q }) \right)$\\
						\hline
						$p\in \mathcal{N}(\mathcal{H})$ &  $ cp(q+1)\leq s<(c+1)p(q+1) $& $\big( c \, SRG^s(2p) \dot{\otimes} SRG^s(2q+2) \big)$\\
						&& $\cup \big( (s-cp(q+1) )K_1,\emptyset \big)$\\
						\hline
						$p\in \mathcal{N}(\mathcal{C})$ &  $ c(p+1)(q+1)\leq s<(c+1)(p+1)(q+1) $& $\big( c \, SRG^s(2p+2) \dot{\otimes} SRG^s(2q+2) \big)$\\
						&& $\cup \big( (s-c(p+1)(q+1) )K_1,\emptyset \big)$\\
						\hline
				\end{tabular}}
			\end{table}
		\end{enumerate}
	\end{Thm}

	\begin{remark}
		Let $\dot{H}$ be a totally disconnected graph of order $s$, $\dot{\mathscr{B}}_m$ denote the maximal signed bipartite graph with $\dot{H}$ as a star complement for an eigenvalue $\mu$, where there are positive integers $p$ and $ q$ ($p>q$) such that $\mu^2=p q\geq 2$. Then by Theorem \ref{C2}, the three following results are obtained.
		\begin{enumerate}
			\item[(1)] If $p\in \mathcal{N}(\mathcal{H})$, $q\in \mathcal{N}(\mathcal{C})$ and $s=pq$, then $\dot{\mathscr{B}}_m\cong^s SRG^s(2p)\dot{\otimes} (K_{1,q}, \emptyset)$.
			\item[(2)] If $p,\, q\in \mathcal{N}(\mathcal{C}) $ and $s=pq$, then $\dot{\mathscr{B}}_m \cong^s \left(K_{1,pq}, \emptyset \right)$.
			\item[(3)] If $p,\, q\in \mathcal{N}(\mathcal{C})$ and $s=(p+1)q$, then 
			$\dot{\mathscr{B}}_m \cong^s  SRG^s(2p+2)\dot{\otimes} (K_{1,q}, \emptyset)$.
		\end{enumerate}	
	\end{remark}

	\noindent	$\mathbf{References}$
	\nolinenumbers
	\bibliographystyle{amsplain}

\end{document}